\documentclass[11pt]{article}

\usepackage{amsmath, amssymb, amsthm}
\usepackage{hyperref}
\usepackage{dsfont}
\usepackage{enumitem}

\title{An extension of Stein's method incorporating dependence}

\author{
  Aleksandar~Bala\v{s}ev-Samarski\thanks{Department of Mathematics and Computer Science, Faculty of Science, University of Sarajevo, Bosnia and Herzegovina. Email: aleksandar.bs@pmf.unsa.ba}
  \and
  Abdol-Reza~Mansouri\thanks{Department of Mathematics \& Statistics, Queen's University, Canada. Email: mansouri@queensu.ca}
}

\date{October 31, 2026}

\newcommand{\E}[1]{\mathbb{E} \left[ #1 \right]}

\newcommand{\N}{\mathbb{N}}
\newcommand{\R}{\mathbb{R}}

\newcommand{\ouriff}{\iff}
\newcommand{\ind}{\mathds{1}}
\newcommand{\ourimplies}{\implies}
\newcommand{\dw}{d_{\mathrm{W}}}

\newtheorem{theorem}{Theorem}[section]   
\newtheorem{proposition}[theorem]{Proposition}

\theoremstyle{remark}
\newtheorem{remark}[theorem]{Remark}

\newcommand{\keywords}[1]{\par\noindent
  \textbf{Keywords: } #1}

\newcommand{\msc}[1]{\par\noindent
  \textbf{MSC 2020: } #1}

\begin{document}

\maketitle

\begin{abstract}
We extend Stein's method to include a prescribed dependence with respect to an auxiliary random variable, for conditional laws for which Stein characterizations do exist. 
This considerably extends the range of applications of Stein's method. 
The characterization we derive subsumes the case of independent random variables.
Using the Malliavin calculus, we illustrate our extension on the example of a conditionally Gaussian/Beta-distributed random variable. 
\end{abstract}

\keywords{Stein's method, Stein operators, Malliavin--Stein method, Malliavin calculus
}

\msc{60B10, 60H07}

\section{Introduction}

Stein's method is a powerful tool for deriving quantitative bounds on the distance between the law of a given random variable and a target law, and is based on the characterization of the target distribution through a suitable operator acting on a suitable class of test functions \cite{S1972, S1986, NP2012}.
\par
To make matters more precise, let the target probability law on $\R$ be denoted by $\mu$.
Stein's method consists in characterizing the target distribution $\mu$ by constructing a suitable linear operator $\mathcal{N}$ (the ``Stein operator'' for the target distribution $\mu$) acting on a suitable class of functions $\mathcal{C}$, such that the law of $X$ is equal to $\mu$ if and only if
$ \E{\mathcal{N} \! f(X)} = 0$ for all $f \in \mathcal{C}$.
\par
This characterization of the target law $\mu$ by the Stein operator $\mathcal{N}$ can in turn be used to provide bounds on the distance between the law of a given random variable $X$ and $\mu$, using the device known as Stein's equation, given by
\[
\mathcal{N} \! f(x) = h(x) - \E{h(M)}, \quad x \in \R,
\] 
where $h$ is assumed to belong to a suitably rich class
of Borel-measurable functions, and where $M$ is a random variable on $(\Omega, \mathcal{F}, \mathbb{P})$ having the desired
law $\mu$.
The fundamental intuition here is that if the random variable $X$ has law close to that of $M$, i.e., to the target law $\mu$, then the expectation $\E{\mathcal{N}\!f(X)}$ should be ``close'' to $0$, for any function $f$ satisfying the Stein equation. This general idea can be fruitfully applied to Malliavin-differentiable random variables, yielding a characterization of the target law in terms of Malliavin derivatives and other related operators, leading to the extension known as the Malliavin--Stein method~\cite{NP2012}.
\par
In recent work~\cite{P2022, T2025, ABSARM2026}, Stein's method has been extended to yield a characterization for a random variable $X$ to have a desired law $\mu$ and to be independent of an auxiliary random variable $Y$.
More specifically, for $M$ a random variable with the desired law $\mu$ and $Y$ an auxiliary random variable,  a Stein-type characterization is provided for the joint distribution $\mathbb{P}_{X,Y}$ of the random variables $X,Y$ to equal the product distribution $\mathbb{P}_M \otimes \mathbb{P}_Y$.
Such a characterization was initially introduced for normally distributed random variables in~\cite{P2022}. This characterization was then extended to normally distributed random vectors, including extensions to joint asymptotic normality and independence for Malliavin-differentiable random variables~\cite{T2025}, with further extensions to Gamma-distributed random variables in~\cite{T2024}, and finally to arbitrary laws in \cite{ABSARM2026}. 
\par
In this work, we consider the problem of deriving a Stein characterization for a given random variable $X$ to have a \emph{prescribed dependence} on an auxiliary random variable $Y$. For simplicity of presentation and without loss of generality, we shall assume all random variables to be real-valued.
To make matters precise, let $M$ be a random variable that has the prescribed dependence on the random variable $Y$. 
With $\mu_M$ denoting the law of $M$ and $\mu_Y$ that of $Y$, and
$\mathring{Y}(\Omega)$ denoting the essential range of $Y$,
the dependence of $M$ on $Y$ is captured by the family $(\nu_y)_{y \in \mathring{Y}(\Omega)}$ of conditional probability measures on the Borel sigma-algebra $\mathcal{B}_\R$ of $\R$, which are defined by the relations 
\[
    \mathbb{P}(M \in A, Y \in B) = \int_B \nu_y(A) \, \mu_Y(\mathrm{d}y) \quad (\forall A, B \in \mathcal{B}_\R).
\]
The existence of such conditional probability measures is guaranteed in our setting of real-valued random variables, as well as in substantially more general ones.
We refer the interested reader to \cite{D2002, K2021} for a detailed treatment of conditional probability measures and regular conditional probabilities.
\par
We assume given a suitable class $\mathcal{C}$ of functions, and for each $y \in \mathring{Y}(\Omega)$, we assume given a Stein operator $\mathcal{N}_y$ for the corresponding conditional law $\nu_y$, acting on $\mathcal{C}$; that is, we assume that for any random variable $Z$, we have, for all $y \in \mathring{Y}(\Omega)$:
\[
    Z \sim \nu_y \ouriff \E{ \mathcal{N}_y f(Z)} = 0 \quad (\forall f \in \mathcal{C}).
\]
Our goal in this paper is to derive a Stein characterization for the joint law $\mathbb{P}_{X,Y}$ of $X, Y$ to equal the joint law $\mathbb{P}_{M,Y}$ of $M, Y$.
It is in this precise sense that $X$ will have the same dependence on $Y$ as $M$ does.
To the best of our knowledge, this is the first such extension of Stein's method.
\par
We wish to point out that even though the method of proof we use in this work to derive our characterization is inspired in large part by the method we introduced in our earlier work \cite{ABSARM2026} to circumvent the use of characteristic functions for the case of independent random variables, the characterization we derive in this work is by no means a trivial consequence of earlier results.
\par
With the Stein characterization thus obtained, we then proceed to estimate Wasserstein and total variation distances between the joint laws $\mathbb{P}_{X,Y}$ and $\mathbb{P}_{M,Y}$.
Note that if $M$ and $Y$ are assumed independent, then all the conditional probability measures $\nu_y$ are equal, and the joint law $\mathbb{P}_{M,Y}$ of $M,Y$ is then given by the product measure $\mathbb{P}_M \otimes \mathbb{P}_Y$.
Our Stein characterization in this very specific case then reduces to the one obtained in \cite{ABSARM2026}.
\par
Finally, we treat the case of a binary random variable $Y$, with the prescribed dependence of $M$ on $Y$ being standard normal, respectively Beta-distributed, depending on the values of $Y$, and we apply our Stein characterization to that specific example.

\section{Main results}
\label{sec:Main}

Let now $M, Y$ be random variables on $(\Omega, \mathcal{F}, \mathbb{P})$, with respective laws $\mu_M$ and $\mu_Y$, and with the dependence of $M$ on $Y$ given by the family $(\nu_y)_{y \in \mathring{Y}(\Omega)}$ of probability laws, where
$\mathring{Y}(\Omega)$ is the essential range of $Y$, and where
\[
    \mathbb{P}(M \in A, Y \in B) = \int_B \nu_y(A) \, \mu_Y(\mathrm{d}y) \quad (\forall A, B \in \mathcal{B}_\R).
\]
We assume given a Stein operator $\mathcal{N}_y$ for the law $\nu_y$ for
all $y \in \mathring{Y}(\Omega)$, as well as a suitable common class $\mathcal{C}$ of functions for the family of operators $(\mathcal{N}_y)_{y \in \mathring{Y}(\Omega)}$. 
We also assume given a rich enough class $\mathcal{H}$ of Borel-measurable functions on $\R$ such that for any $h \in \mathcal{H}$, and for any $y \in \mathring{Y}(\Omega)$, the Stein equation
\[
\mathcal{N}_y f = h - \int_\R h \, \mathrm{d} \nu_y
\]
has a unique solution in $\mathcal{C}$.
\par
Let now $X$ be a random variable on $(\Omega, \mathcal{F}, \mathbb{P})$. As stated earlier, we wish to derive a Stein characterization for $X$ to have the same dependence on $Y$ as $M$ does, that is, we are looking for a Stein-like characterization for the joint law $\mathbb{P}_{X,Y}$ to equal the joint law $\mathbb{P}_{M,Y}$. Even though for notational simplicity we assume all our random variables to be $\R$-valued, our results can be trivially extended to the case where they are vector-valued.
\par
With the notation in place, the following preliminary result yields a necessary condition for $\mathbb{P}_{X,Y}$ and $\mathbb{P}_{M,Y}$ to be equal:
\par

\begin{proposition}\label{pr:1}
Assume the mapping $(x,y) \mapsto \mathcal{N}_yf(x)$ is $\mathbb{P}_{M,Y}$-integrable on $\R^2$ for all $f \in {\mathcal{C}}$. 
We then have:
\begin{equation}\label{eq:prop1}
    \mathbb{P}_{X,Y} = \mathbb{P}_{M,Y} \ourimplies
    \E{\mathcal{N}_Y f(X) \,\middle|\, \sigma(Y)} = 0 \ \mathrm{a.s.} \quad (\forall f \in \mathcal{C}).  
\end{equation}
\end{proposition}

\begin{proof}

    Assume that $\mathbb{P}_{X,Y} = \mathbb{P}_{M,Y}$. We then have, for any $\mathbb{P}_{X,Y}$-integrable function
    $g \colon \R \times \R \to \R$,
    \[
    \int_{\R^2} g \, \mathrm{d}\mathbb{P}_{X,Y} = 
    \int_\R \left(\int_\R g(x,y) \nu_y(\mathrm{d}x)\right) \mu_Y(\mathrm{d}y).
    \]
    Let now $k \colon \R \to \R$ be Borel-measurable and bounded. For every
    $f \in \mathcal{C}$, we have, by integrability of $(x,y) \mapsto \mathcal{N}_y f(x)$, 
\begin{align*}
\E{k(Y) \mathcal{N}_Y f(X)} &= \int_\R \left(\int_\R k(y) \mathcal{N}_y f(x) \nu_y(\mathrm{d}x)\right) \mu_Y(\mathrm{d}y) \\
    &= 
\int_\R k(y) \left( \int_\R \mathcal{N}_y f(x) \nu_y(\mathrm{d}x) \right) \mu_Y(\mathrm{d}y) = 0,
\end{align*}

    and since $k$ was arbitrary, it follows that 
    \[
    \E{\mathcal{N}_Y f(X) \,\middle|\, \sigma(Y)} = 0 \ \mathrm{a.s.} \quad (\forall f \in \mathcal{C}).
    \]
\end{proof}

In what follows, we denote by $\tilde{\mathcal{C}}$ the class of functions $f \colon \R \times \R \to \R$, $(x,y) \mapsto f(x,y)$, for which $f(\cdot,y) \in \mathcal{C}$ for all $y \in \R$.
In the expression $\mathcal{N} \! f(x,y)$, the operator $\mathcal{N}$ is assumed to act only on the first argument.
\par
With the random variable $Y$ having essential range 
$\mathring{Y}(\Omega)$, we define 
the class $\tilde{\mathcal{H}}$ to be the class of functions $h \colon \R \times \R \to \R$, such that for all $y \in \mathring{Y}(\Omega)$, the function $h_y \colon \R \to \R$ defined by $h_y(x) = h(x,y)$ for all $x \in \R$, belongs to $\mathcal{H}$, and such that $h(x,y) = 0$ for all $x \in \R$ and $y \in \R \setminus \mathring{Y}(\Omega)$. 
We assume the subset of $\tilde{\mathcal{H}}$ consisting of bounded functions to be \emph{separating} on $\R^2$ for the family of joint laws $(\mathbb{P}_{Z,Y})_Z$ indexed by all $\R$-valued random variables on $(\Omega,\mathcal{F},\mathbb{P})$. 
\par
Let $h \in \tilde{\mathcal{H}}$, and consider the equation
\begin{equation}\label{eq:MSE}
    \mathcal{N}_y  f_y(x) = h_y(x) - \int_\R h_y(z) \nu_y(\mathrm{d}z), \quad x \in \R,
\end{equation}
where $y \in \R$.
For all $y \in \mathring{Y}(\Omega)$, the Stein equation \eqref{eq:MSE} has a unique solution in $\mathcal{C}$, which we denote by $f_{y,h}$.
Defining 
\[
    f_h \colon \R \times \R \to \R, \ (x,y) \mapsto f_h(x,y) 
\]
by 
\[
    f_h(x,y) = \begin{cases} 
        f_{y,h}(x), & x \in \R, \ y \in\mathring{Y}(\Omega) \\
        0,          & x \in \R, \ y \notin\mathring{Y}(\Omega)
    \end{cases},
    \]
we can write 
\begin{equation}\label{eq:PSE}
    \mathcal{N}_y f_{h}(x,y) = h(x,y) - \int_\R h(z,y) \nu_y(\mathrm{d}z), \quad \forall x,y \in \R.
\end{equation}
Note that we have chosen $f_h(x,y)= 0$ for $y \notin \mathring{Y}(\Omega)$ for convenience. 

Note also that $(x,y) \mapsto \mathcal{N}_y f_h(x,y)$ is required to be jointly measurable in $(x,y)$, an assumption that is verified in almost all cases of interest.

With the classes $\tilde{\mathcal{C}}$ and $\tilde{\mathcal{H}}$ as defined above, we can now state our main result:
\begin{theorem}\label{th:1}
Let $M, Y$ be random variables on $(\Omega, \mathcal{F}, \mathbb{P})$, with respective laws $\mu_M$ and $\mu_Y$, and with the dependence of $M$ on $Y$ given by the family $(\nu_y)_{y \in \mathring{Y}(\Omega)}$ of probability laws, where
$\mathring{Y}(\Omega)$ is the essential range of $Y$, and where
\[
    \mathbb{P}(M \in A, Y \in B) = \int_B \nu_y(A) \, \mu_Y(\mathrm{d}y) \quad (\forall A, B \in \mathcal{B}_\R).
\]
We assume given a Stein operator $\mathcal{N}_y$ for the law $\nu_y$ for
all $y \in \mathring{Y}(\Omega)$, as well as a suitable common class $\mathcal{C}$ of functions for the family of operators $(\mathcal{N}_y)_{y \in \mathring{Y}(\Omega)}$. We assume the mapping $(x,y) \mapsto \mathcal{N}_yf(x,y)$ to be $\mathbb{P}_{M,Y}$-essentially bounded on $\R^2$ for all $f \in \tilde{\mathcal{C}}$ and continuous in $y$.
Let $X$ be a $\R$-valued random variable on $(\Omega, \mathcal{F}, \mathbb{P})$.
We then have:
\[
    \mathbb{P}_{X,Y} = \mathbb{P}_{M,Y} \ouriff 
    \E{\mathcal{N}_Y f(X,Y) } = 0 \quad (\forall f \in \tilde{\mathcal{C}}).  
\]
\end{theorem}

\begin{remark}
In the case where $\mathbb{P}_{M,Y} = \mathbb{P}_M \otimes \mathbb{P}_Y$, i.e., $M$ and $Y$ are assumed to be independent, all the Stein operators $\mathcal{N}_y$ ($y \in \mathring{Y}(\Omega)$) are then equal to the same Stein operator $\mathcal{N}$, and
the characterization obtained here reduces to the one in \cite{ABSARM2026}.
Similarly, in that case, equation \eqref{eq:MSE} reduces to equation (2) of \cite{ABSARM2026}. 
\end{remark}

\begin{proof}[Proof of Theorem~\ref{th:1}]

($\Longrightarrow$): Assume $\mathbb{P}_{X,Y} = \mathbb{P}_{M,Y}$.
Assume first that $Y$ has finite range, and let $Y(\Omega) = \lbrace y_i \rbrace_{i=1}^N$.
Without loss of generality, we assume $\mathbb{P}(Y = y_i) > 0$ for all $i$.
Let now $f \in \tilde{\mathcal{C}}$.
We have, for all $x \in \R$ and for $\mu_Y$-almost all $y \in \R$:
\[
    f(x,y) = \sum_{i=1}^N f(x,y_i) \ind_{\lbrace y_i \rbrace}(y).
\]
Denote by $f_{y_i}$ the mapping $x \mapsto f(x,y_i)$, for each $i$. 
Then $f_{y_i} \in \mathcal{C}$ for all $i \in \lbrace 1,\ldots,N \rbrace$, and we 
have for all $x \in \R$ and for $\mu_Y$-almost all $y \in \R$:
\[
   f(x,y) = \sum_{i=1}^N f_{y_i}(x) \ind_{\lbrace y_i \rbrace}(y).
\]
It follows from Proposition~\ref{pr:1} that
\[
\E{\mathcal{N}_y f(X,y) \,\middle|\, \sigma(Y)} = 0 \ \mathrm{a.s.}, \ \forall y \in \R.
\]
We have:
\[
\E{\mathcal{N}_y f(X,y) \,\middle|\, \sigma(Y)} = \sum_i c_i(y) \ind_{\{Y=y_i\}},
\]
where the $c_i$ are functions depending on $y$ (as well as on $f$), and hence we have $c_i(y)=0$ for all $y \in \R$ and for all $i$. We also have:
\[
\E{\mathcal{N}_Y f(X,Y) \,\middle|\, \sigma(Y)} = \sum_i d_i \ind_{\{Y=y_i\}},
\]
and hence $d_i = c_i(y_i) = 0$ for all $i$. 
Therefore,
\[
\E{\mathcal{N}_Y f(X,Y) \,\middle|\, \sigma(Y)} = 0 \ \mathrm{a.s.}
\]
for all $f \in \tilde{\mathcal{C}}$, and it follows that
\[
    \E{ \mathcal{N}_Y f(X,Y)} = \E{ \E{ \mathcal{N}_Y f(X,Y) \,\middle|\, \sigma(Y)}} = 0. 
\]
We have thus shown that for $Y$ having finite range,
\[
\mathbb{P}_{X,Y} = \mathbb{P}_{M,Y} \ourimplies \E{ \mathcal{N}_Y f(X,Y)} = 0 \quad (\forall f \in \tilde{\mathcal{C}}).
\]
For general $Y$, consider a sequence $(Y_n)_{n \in \N}$ of $\sigma(Y)$-measurable simple functions converging almost surely to $Y$, and let $f \in \tilde{\mathcal{C}}$. 
The sequence $(\mathcal{N}_{Y_n}f(X,Y_n))_{n \in \N}$ then converges almost surely to $\mathcal{N}_{Y}f(X,Y)$ on $\Omega$, and the essential boundedness assumption yields uniform integrability of the sequence $(\mathcal{N}_{Y_n}f(X,Y_n))_{n \in \N}$, and hence we obtain
\[
\E{ \mathcal{N}_Y f(X,Y)} = \E{ \lim_n \mathcal{N}_{Y_n} f(X,Y_n)} = \lim_n \E{ \mathcal{N}_{Y_n}f(X,Y_n)} = 0.
\]
\par
($\Longleftarrow$): 
Let $(\bar{\nu}_y)_{y \in \mathring{Y}(\Omega)}$ be the family of conditional probability measures of $X$ given $Y$, i.e., such that for any   $\mathbb{P}_{X,Y}$-integrable function $g\colon\R^2 \to \R$, we have
\[
    \int_{\R^2} g \, \mathrm{d}\mathbb{P}_{X,Y} = 
    \int_\R \left(\int_\R g(x,y) \bar{\nu}_y(\mathrm{d}x)\right) \mu_Y(\mathrm{d}y).
    \]
Let $h$ be an arbitrary bounded element of $\tilde{\mathcal{H}}$, and let $f_h \in \tilde{\mathcal{C}}$ be the unique solution to \eqref{eq:PSE}. We have:
\begin{align*}
   0=\E{ \mathcal{N}_Y f_h(X,Y)} &= 
    \int_\R \left ( \int_\R \mathcal{N}_y f_h(x,y) \bar{\nu}_y(\mathrm{d}x) \right) \mu_Y(\mathrm{d} y) \\
    &= \int_{\R^2} \left(h(x,y) - \int_\R h(z,y) \nu_y (\mathrm{d}z) \right) \bar{\nu}_y(\mathrm{d} x) \mu_Y(\mathrm{d}y) \\
    &= \int_{\R^2} h(x,y) \bar{\nu}_y(\mathrm{d}x) \mu_Y(\mathrm{d} y) -
    \int_{\R^2} h(z,y) \nu_y(\mathrm{d}z) \mu_Y(\mathrm{d} y) \\
    &= \mathbb{E}^{\mathbb{P}_{X,Y}}[h] - \mathbb{E}^{\mathbb{P}_{M,Y}}[h],
\end{align*}
and since the bounded element $h \in \tilde{\mathcal{H}}$ was assumed arbitrary, and the subset of
$\tilde{\mathcal{H}}$ consisting of bounded functions was assumed separating, the result follows.
\end{proof}

\subsection{Bounds on the total variation and Wasserstein distances between laws using Stein equations}\label{subs:32}

Choosing an appropriate class $\tilde{\mathcal{E}}$ that contains all the solutions of \eqref{eq:PSE} for $h \in \tilde{\mathcal{H}}$ (and possibly more), we obtain
\begin{equation}\label{eq:PSE1}
    \sup_{h \in \tilde{\mathcal{H}}} \left| \E{h(X,Y)} - \E{h(M,Y)} \right| \leq 
    \sup_{f \in \tilde{\mathcal{E}}} \left| \E{\mathcal{N}_Y \! f(X,Y)} \right|.
\end{equation}

Let then $\tilde{\mathcal{H}}_{\mathrm{TV}}$ denote the class of all
indicator functions of Borel subsets of $\R^2$, and let
$\tilde{\mathcal{E}}_{\mathrm{TV}}$ be a class of functions that contains all the solutions of \eqref{eq:PSE} for $h \in \tilde{\mathcal{H}}_{\mathrm{TV}}$; similarly, let $\tilde{\mathcal{H}}_{\mathrm{W}}$ denote the class of all $1$-Lipschitz functions on $\R^2$, and let
$\tilde{\mathcal{E}}_{\mathrm{W}}$ be a class of functions that contains all the solutions of \eqref{eq:PSE} for $h \in \tilde{\mathcal{H}}_{\mathrm{W}}$; we then obtain immediately:
\begin{theorem}\label{th:TVandWBound}
Let $M, Y$ be random variables on $(\Omega, \mathcal{F}, \mathbb{P})$, with respective laws $\mu_M$ and $\mu_Y$, and with the dependence of $M$ on $Y$ given by the family $(\nu_y)_{y \in \mathring{Y}(\Omega)}$ of probability laws, where
$\mathring{Y}(\Omega)$ is the essential range of $Y$, and where
\[
    \mu_M(A) = \int_\R \nu_y(A) \mu_Y(\mathrm{d}y), \quad (\forall A \in \mathcal{B}_\R).
\]
We assume given a Stein operator $\mathcal{N}_y$ for the law $\nu_y$ for
all $y \in \mathring{Y}(\Omega)$, as well as a suitable common class $\mathcal{C}$ of functions for the family of operators $(\mathcal{N}_y)_{y \in \mathring{Y}(\Omega)}$. Let $X$ be a given random variable.
With $d_{\mathrm{TV}}$ denoting the total variation distance, and
$d_{\mathrm{W}}$ the Wasserstein distance between laws, we have:
\begin{align*} 
    d_{\mathrm{TV}} (\mathbb{P}_{X,Y}, \mathbb{P}_{M,Y}) & \leq
    \sup_{f \in \tilde{\mathcal{E}}_{\mathrm{TV}}} \E{\mathcal{N}_Y f(X,Y)},
    \\
    d_{\mathrm{W}} (\mathbb{P}_{X,Y}, \mathbb{P}_{M,Y}) & \leq
    \sup_{f \in \tilde{\mathcal{E}}_{\mathrm{W}}} \E{\mathcal{N}_Y f(X,Y)}.
\end{align*} 
\end{theorem}

\section{Applications}

We illustrate the results obtained in the previous section for the case where the auxiliary random variable $Y$ is binary-valued and the reference random variable $M$ has a prescribed dependence on $Y$, being conditionally normal $\mathcal{N}(0,1)$ and $\mathrm{Beta}(\frac{1}{2},1)$-distributed with respect to $Y$.
\par
We thus assume given an isonormal Gaussian process $W=\{W(\mathfrak{h}): \mathfrak{h} \in \mathfrak{H}\}$ defined on a suitable probability space $(\Omega,\mathcal{F},\mathbb{P})$, where $\mathfrak{H}$ is a real separable Hilbert space with inner product $\langle \cdot,\cdot \rangle_{\mathfrak{H}}$.
We denote by $\mathbb{D}^{1,2}$ the domain in $L^2(\Omega)$ of the Malliavin derivative operator $D$, by $\delta$ the adjoint of $D$, and by $L=-\delta D$ the infinitesimal generator of the corresponding Ornstein--Uhlenbeck semigroup \cite{N2006}.
Let $(\mathfrak{h}_i)_{i \geq 0}$ be a Hilbert basis for $\mathfrak{H}$, and let
the random variables $Y$ and $M$ be defined by
\[
    Y = \begin{cases}
        y_1, & I_1(\mathfrak{h}_0) \geq \alpha\\
        y_2, & \text{otherwise}
    \end{cases},
\]
with $\alpha,y_1,y_2 \in \R$ ($y_1 \neq y_2$), and
\[
M = F_1 \ind_{\lbrace Y = y_1 \rbrace} + F_2 \ind_{\lbrace Y = y_2 \rbrace},
\]
where 
$
F_1 = I_1(\mathfrak{h}_1),
$
and
$
F_2 = \exp\left({-I_1(\mathfrak{h}_2)^2 - I_1(\mathfrak{h}_3)^2}\right).
$
\par
Let $p_1 = \mathbb{P}(Y = y_1)$ and $p_2 = 1 - p_1$.
Note that $Y,F_1,F_2$ are independent, with
$F_1 \sim \mathcal{N}(0,1)$ and
$F_2 \sim \mathrm{Beta}(\frac{1}{2},1)$.
Note also that 
$M$ and $Y$ are dependent, with the law of $M$ given $Y=y_1$ being standard normal, and the law of $M$ given
$Y = y_2$ being the $\mathrm{Beta}(\frac{1}{2},1)$ law.
\par
Consider now the sequence $(X_n)_{n \in \N}$ of random variables defined by
\[
X_n = F_1^n \ind_{\lbrace Y = y_1 \rbrace} + F_2^n \ind_{\lbrace Y = y_2 \rbrace}, 
\]
where
$
    F_1^n = \mu_n + \sigma_n I_1(\mathfrak{h}_4),
$
and
$
    F_2^n = \exp\left({-I_1(\mathfrak{h}_5)^2 - \frac{n}{n+1}I_1(\mathfrak{h}_6)^2}\right),
$
with $\mu_n \to 0$ and $\sigma_n \to 1$ as $n \to \infty$.
As a direct application of the bound on the Wasserstein distance obtained in Theorem~\ref{th:TVandWBound}, we now prove that $\dw(\mathbb{P}_{X_n,Y}, \mathbb{P}_{M,Y}) \to 0$ as $n \to \infty$, that is, $X_n$ has asymptotically the same dependence on $Y$ as $M$ does. 
A finer analysis would yield the precise rate at which this convergence happens.
\par
Recall that, by Theorem~\ref{th:TVandWBound},  
\begin{equation}\label{eq:ex1}
    \dw(\mathbb{P}_{X_n,Y},\mathbb{P}_{M,Y}) \leq \sup_{f \in \tilde{\mathcal{E}}_{\mathrm{W}}} \left|\E{ \mathcal{N}_Y f(X_n,Y)}\right|.
\end{equation}
We have $Y(\Omega) = \lbrace y_1, y_2 \rbrace$, and it follows from the very definition of $M$ that the Stein operators for $Y = y_1$ and $Y = y_2$ are given respectively by
\[
    \mathcal{N}_{y_1}f(x,y) = 
    \partial_x f(x,y) - xf(x,y), \quad (x,y) \in \R^2,
\]
and
\[
    \mathcal{N}_{y_2}f(x,y) = 
    \frac{2}{3} \,x (1-x) \partial_x f(x,y) - \left( x - \frac{1}{3} \right) f(x,y), \quad x \in [0,1], \ y \in \R.
\]
We refer to \cite{KT2012} for more details on this.
We have, for all $n \in \N$:
\begin{align*}
    \E{\mathcal{N}_Y f(X_n,Y)} 
        &= \E{\mathcal{N}_Y f(X_n,Y) \ind_{\lbrace Y = y_1 \rbrace}} + 
            \E{\mathcal{N}_Y f(X_n,Y) \ind_{\lbrace Y = y_2 \rbrace}} \\
        &= \E{\mathcal{N}_{y_1} f(F_1^n,y_1) \ind_{\lbrace Y = y_1 \rbrace}} +
           \E{\mathcal{N}_{y_2} f(F_2^n,y_2) \ind_{\lbrace Y = y_2 \rbrace}} \\
        &= p_1 \E{\mathcal{N}_{y_1} f(F_1^n,y_1)} + 
            p_2 \E{\mathcal{N}_{y_2} f(F_2^n,y_2)},
\end{align*}
where we used the independence of $F_1^n, F_2^n$ from $Y$. We now need to estimate 
$\left\vert \E{\mathcal{N}_{y_i} f(F_i^n,y_i)} \right\vert$ 
 for $i=1,2$, where $f=f_h$ is the solution to the corresponding Stein equation \eqref{eq:PSE} for $1$-Lipschitz $h$.
\par
We have, for all $n \in \N$:
\begin{align*}
    \E{\mathcal{N}_{y_1} f(F_1^n,y_1)}
        &= \E{\partial_x f(F_1^n,y_1) - F_1^nf(F_1^n,y_1)} \\
        &= \E{\partial_x f(F_1^n,y_1) - (F_1^n - \mu_n)f(F_1^n,y_1)} - \mu_n\E{f(F_1^n,y_1)},
\end{align*}
and using the fact that $F_1^n \in \mathbb{D}^{1,2}$ and $\E{F_1^n - \mu_n}=0$, we obtain, applying the Malliavin integration by parts formula:
\begin{align*}
    \E{(F_1^n - \mu_n) f(F_1^n,y_1)} 
        &= \E{(\delta D)(-L)^{-1} (F_1^n - \mu_n) f(F_1^n,y_1) } \\
        &= \E{ \partial_x f(F_1^n,y_1) \langle D F_1^n, D(-L)^{-1}(F_1^n - \mu_n) \rangle_\mathfrak{H}}.
\end{align*}
As $DF_1^n = D(-L)^{-1}(F_1^n - \mu_n) =  \sigma_n \mathfrak{h}_4$,
we obtain $$\langle DF_1^n, D(-L)^{-1}(F_1^n - \mu_n) \rangle_\mathfrak{H} = \sigma_n^2,$$ and hence,
\[
\E{(F_1^n-\mu_n) f(F_1^n,y_1)} = \sigma_n^2 \E{ \partial_x f(F_1^n,y_1)}.
\]
Therefore
\[
\E{\mathcal{N}_{y_1} f(F_1^n,y_1)} = (1-\sigma_n^2) \E{ \partial_x f(F_1^n,y_1)} 
- \mu_n\E{f(F_1^n,y_1)}.
\]
\par
Similarly, we obtain after a slight rearrangement that, for all $n \in \N$:
\begin{multline*}
   \E{\mathcal{N}_{y_2} f(F_2^n,y_2)} = 
   \E{\frac{2}{3} F_2^n(1-F_2^n)\partial_x f(F_2^n,y_2) -
   \left( F_2^n - \E{F_2^n} \right) f(F_2^n,y_2)} \\
   - \E{ \left( \E{F_2^n}-\frac{1}{3} \right) f(F_2^n,y_2)},
\end{multline*}
and, applying the Malliavin integration by parts formula, we obtain:
\begin{align*}
 \E{(F_2^n - \E{F_2^n}) f(F_2^n,y_2)} &= \E{(\delta D) (-L)^{-1} \left( F_2^n - \E{F_2^n} \right) f(F_2^n,y_2)} \\
    &= \E{\partial_x f(F_2^n,y_2) \left\langle DF_2^n,D(-L)^{-1} \left( F_2^n - \E{F_2^n} \right) \right\rangle_{\mathfrak{H}} },
\end{align*}
which yields
\begin{align*}
&\E{\mathcal{N}_{y_2} f(F_2^n,y_2)} \\
&\quad = \E{\partial_x f(F_2^n,y_2) \Bigl( \tfrac{2}{3} F_2^n \bigl(1-F_2^n\bigr)
  - \bigl\langle DF_2^n, D(-L)^{-1}\bigl(F_2^n - \E{F_2^n}\bigr) \bigr\rangle_{\mathfrak{H}} \Bigr)} \\
&\quad\phantom{=}\; - \E{ \bigl( \E{F_2^n} - \tfrac{1}{3} \bigr) f(F_2^n,y_2) }.
\end{align*}
Putting everything together, we obtain:
\begin{multline*}
    \left|\E{\mathcal{N}_Y f(X_n,Y)}\right|
    \leq {}  
    p_1\left(  \left| 1-\sigma_n^2 \right| \E{ \left| \partial_x f(F_1^n,y_1) \right|} 
+ |\mu_n| \E{ \left| f(F_1^n,y_1) \right| }\right) \\
    + p_2\left(\E{\left| \partial_x f(F_2^n,y_2)\left( \frac{2}{3} F_2^n \left( 1-F_2^n \right) - \left\langle DF_2^n,D(-L)^{-1} \left( F_2^n -\E{F_2^n} \right) \right\rangle_{\mathfrak{H}} \right) \right|
   } 
    \right) \\
+ p_2\left(\E{ \left| \left( \E{F_2^n}-\frac{1}{3} \right) f(F_2^n,y_2) \right|}   \right).
\end{multline*}

It follows from the results in \cite{D2012} that there exist constants $C_1,C_2 > 0$ such that
the solution $f_h$ to the Stein equation \eqref{eq:PSE} satisfies
\[
\| f_h \|_{L^\infty(\R^2)} \leq C_1, \quad \| \partial_x f_h \|_{L^\infty(\R^2)} \leq C_2,
\]
for all $h$ $1$-Lipschitz and for $y \in \lbrace y_1, y_2 \rbrace$. Hence, for all $n \in \N$:
\begin{multline*}
\dw(\mathbb{P}_{X_n,Y},\mathbb{P}_{M,Y})
\leq p_1\bigl( C_2 \lvert 1-\sigma_n^2 \rvert + C_1 \lvert \mu_n \rvert \bigr) \\
+ p_2 C_2 \E{ \left\lvert \tfrac{2}{3} F_2^n(1-F_2^n)
  - \bigl\langle DF_2^n, D(-L)^{-1}\bigl(F_2^n - \E{F_2^n}\bigr) \bigr\rangle_{\mathfrak{H}} \right\rvert } \\
+ p_2 C_1 \bigl\lvert \E{F_2^n} - \tfrac{1}{3} \bigr\rvert.
\end{multline*}
Using the results in \cite{ABSARM2026}, we have:
\begin{multline*}
\bigl\langle D(-L)^{-1}\bigl(F_2^n - \E{F_2^n}\bigr), DF_2^n \bigr\rangle_{\mathfrak{H}}
= 4 \exp\Bigl\{ -\Bigl( I_1(\mathfrak{h}_5)^2 + \tfrac{n}{n+1}\, I_1(\mathfrak{h}_6)^2 \Bigr) \Bigr\} \\
\times \int_0^1
\frac{a \Bigl( \bigl( 1 + 2\tfrac{n}{n+1}(1-a^2) \bigr) I_1(\mathfrak{h}_5)^2
      + \tfrac{n^2}{(n+1)^2}(3-2a^2)\, I_1(\mathfrak{h}_6)^2 \Bigr)}
     {(3-2a^2)^{3/2} \bigl( 1 + 2\tfrac{n}{n+1}(1-a^2) \bigr)^{3/2}} \\
\times \exp\biggl\{ -a^2 \biggl( \frac{I_1(\mathfrak{h}_5)^2}{3-2a^2}
      + \frac{I_1(\mathfrak{h}_6)^2}{\tfrac{1}{n} + 3 - 2a^2} \biggr) \biggr\} \, \mathrm{d}a.
\end{multline*}
The random variable
\[
    \hat{F}_2 = \exp \left({-I_1(\mathfrak{h}_5)^2 - I_1(\mathfrak{h}_6)^2} \right)
\]
is $\mathrm{Beta}(\frac{1}{2},1)$-distributed, 
and a direct computation shows that it satisfies the Malliavin--Stein equation 
\[ 
    \left\langle D\hat{F}_2,D(-L)^{-1}\left(\hat{F}_2-\frac{1}{3}\right) \right\rangle_{\mathfrak{H}}  = \frac{2}{3} \hat{F}_2(1-\hat{F}_2) \ \mathrm{a.s.} 
\]
Since $F_2^n \to \hat{F}_2$ pointwise on $\Omega$, it follows from dominated convergence that 
\[ \E{ \left| \frac{2}{3} F_2^n(1-F_2^n) - \left\langle DF_2^n,D(-L)^{-1} \left( F_2^n-\E{F_2^n} \right)  \right\rangle_{\mathfrak{H}} \right| } \to 0 \quad \textrm{as $n \to \infty$.} \]
Finally, $
\left|\E{F_2^n}-\frac{1}{3} \right| \to 0
$
as $n \to \infty$. Since, by assumption,
$\mu_n \to 0$ and $\sigma_n \to 1$ as $n \to \infty$, it follows that
$
\dw(\mathbb{P}_{X_n,Y},\mathbb{P}_{M,Y}) \to 0
$
as $n \to \infty$, which shows that $X_n$ has asymptotically the prescribed dependence on $Y$.

\bibliographystyle{plain} 
\bibliography{references} 

\end{document}